\documentclass[12pt,reqno]{amsart}
\newtheorem{theorem}{Theorem}[section]
\newtheorem{lemma}[theorem]{Lemma}
\newtheorem{proposition}[theorem]{Proposition}
\newtheorem{corollary}[theorem]{Corollary}

\theoremstyle{definition}
\newtheorem{definition}[theorem]{Definition}

\numberwithin{equation}{section}

\begin{document}

\baselineskip=15.5pt

\title[Equivariant bundles and connections]{Equivariant bundles and connections}

\author[I. Biswas]{Indranil Biswas}

\address{School of Mathematics, Tata Institute of Fundamental
Research, Homi Bhabha Road, Mumbai 400005, India}

\email{indranil@math.tifr.res.in}

\author[A. Paul]{Arjun Paul}

\address{School of Mathematics, Tata Institute of Fundamental 
Research, Homi Bhabha Road, Mumbai 400005, India} 

\email{apmath90@math.tifr.res.in}



\begin{abstract}
Let $X$ be a connected complex manifold equipped with a holomorphic action of a complex
Lie group $G$. We investigate conditions under which a principal bundle on $X$
admits a $G$--equivariance structure.
\end{abstract}

\maketitle

\section{Introduction}

Let $G$ a complex Lie group and $X$ a connected complex manifold equipped with a
holomorphic action of $G$
$$
\rho\,:\, G\times X\, \longrightarrow\, X\, .
$$
Let $E_H$ be a holomorphic principal $H$--bundle on $X$, where $H$ is a connected complex Lie 
group. Given these, we construct a short exact sequence of holomorphic vector bundles
$$
0\,\longrightarrow\, \text{ad}(E_H)\,\longrightarrow\,\text{At}_\rho
(E_H)\,\stackrel{q}{\longrightarrow}\, X\times{\mathfrak g}\, \longrightarrow\, 0\, ,
$$
where $\mathfrak g$ is the Lie algebra of $G$ and $\text{ad}(E_H)$ is the adjoint
vector bundle for $E_H$. The above vector bundle $\text{At}_\rho (E_H)$
is a subbundle of the vector bundle $\text{At}(E_H)\oplus (X\times{\mathfrak g})$,
where $\text{At}(E_H)$ is the Atiyah bundle for $E_H$ (see \eqref{e8}).
A holomorphic $G$--connection on $E_H$ is defined to be a holomorphic splitting of the above
exact sequence. In other words, a holomorphic $G$--connection on $E_H$ is a holomorphic homomorphism
of vector bundles from the trivial vector bundle with fiber $\mathfrak g$
$$
h\, :\, X\times{\mathfrak g}\,\longrightarrow\,\text{At}_\rho(E_H)
$$
such that $q\circ h\,=\, \text{Id}_{X\times{\mathfrak g}}$.

Consider the pullback $\rho^*E_H$ as a holomorphic family of principal $H$--bundles on $X$
parametrized by $G$. Let
$$
\mu\, :\, {\mathfrak g} \,=\, T_eG\, \longrightarrow\, H^1(X,\, {\rm ad}(E_H))
$$
be the infinitesimal deformation map for this family, where $e\,\in\, G$ is the identity element.
We prove that $E_H$ admits a holomorphic $G$--connection if and only if $\mu\,=\,0$
(see Lemma \ref{lem1}).

Let $\mathcal G$ be the complex Lie group consisting of all pairs of the form $(y\, , z)$, where
\begin{itemize}
\item $z\, \in\, G$ such that the holomorphic principal $H$--bundle $(\rho^* E_H)\vert_{\{z\}\times X}
\,\longrightarrow\, X$ is holomorphically isomorphic to $E_H$, and

\item $y\, :\, E_H \,\longrightarrow\, (\rho^* E_H)\vert_{\{z\}\times X}$ is a holomorphic
isomorphism of principal $H$--bundles.
\end{itemize}
This group $\mathcal G$ has natural actions on both $X$ and $E_H$. We show that $E_H$ has
a tautological holomorphic $\mathcal G$--connection whose curvature vanishes identically
(see Proposition \ref{prop2}). From this proposition it follows that every $G$--equivariant
principal $H$--bundle on $X$ has a holomorphic $G$--connection whose curvature vanishes
identically (Lemma \ref{lem4}).

Finally, assume that $G$ is semisimple and simply connected. Then $E_H$ admits a
$G$--equivariant structure if and only if $E_H$ admits a holomorphic $G$--connection
(Theorem \ref{thm1}). Theorem \ref{thm1} implies that $E_H$ admits a
$G$--equivariant structure if and only if the holomorphic principal $H$--bundle
$(\rho^* E_H)\vert_{\{z\}\times X}$ over $X$ is holomorphically isomorphic to $E_H$ for
all $z\, \in\, G$ (Corollary \ref{cor1}).

Equivariant bundles with invariant connections
are investigated in \cite{BU} using a Jordan algebraic approach.

\section{Atiyah exact sequence and group action}

\subsection{Atiyah exact sequence}

Let $H$ be a connected complex Lie group; its Lie algebra will be denoted by 
$\mathfrak h$. Let $X$ be a connected complex manifold; its holomorphic
tangent bundle will be denoted by $TX$. Take a holomorphic 
principal $H$--bundle on $X$
\begin{equation}\label{e-1}
p\, :\, E_H\, \longrightarrow\, X\, .
\end{equation}
Let
\begin{equation}\label{e3}
\psi\, :\, E_H\times H\, \longrightarrow\, E_H
\end{equation}
be the action of $H$ on $E_H$. We note that $p\circ \psi\,=\, p\circ p_{E_H}$, where
$p_{E_H}\,:\, E_H\times H\, \longrightarrow\,E_H$ is the natural projection, and the
action of $H$ on each fiber of $p$ is transitive and free. Let
\begin{equation}\label{e-2}
dp\, :\, TE_H\, \longrightarrow\, p^*TX
\end{equation}
be the differential of $p$, where $TE_H$ is the holomorphic
tangent bundle of $E_H$. Its kernel
$$
T_{\rm rel}\, :=\, \text{kernel}(dp)\, \subset\, TE_H
$$
is known as the \textit{relative tangent bundle} for $p$. So we get a short exact
sequence of holomorphic vector bundles on $E_H$
\begin{equation}\label{e4}
0\,\longrightarrow\, T_{\rm rel}\,\longrightarrow\,TE_H
\,\stackrel{dp}{\longrightarrow}\, p^*TX\, \longrightarrow\, 0\, .
\end{equation}
The differential of $\psi$
in \eqref{e3} produces a homomorphism from the trivial vector bundle on $E_H$ with
fiber $\mathfrak h$
$$
E_H\times{\mathfrak h}\, \longrightarrow\, TE_H
$$
which identifies $T_{\rm rel}$ with $E_H\times{\mathfrak h}$.

The action $\psi$ in \eqref{e3} produces an action of
$H$ on the total space of $TE_H$. The quotient
$$
\text{At}(E_H)\,:=\, (TE_H)/H
$$
is a holomorphic vector bundle on $E_H/H\,=\, X$, which is known as the
Atiyah bundle \cite{At}. Let
\begin{equation}\label{f1}
\text{ad}(E_H)\, :=\, E_H\times^H \mathfrak h\, \longrightarrow\, X
\end{equation}
be the adjoint vector bundle for $E_H$ which is associated to it for the adjoint
action of $H$ on $\mathfrak h$. The action of $H$ on $TE_H$ preserves the
subbundle $T_{\rm rel}$. Using the above identification of $T_{\rm rel}$
with $E_H\times{\mathfrak h}$, we have
$$
T_{\rm rel}/H \,=\, \text{ad}(E_H)\, .
$$
So after taking quotient by $H$, the exact sequence in \eqref{e4} produces a
short exact sequence of holomorphic vector bundles on $X$
\begin{equation}\label{e5}
0\,\longrightarrow\, \text{ad}(E_H)\,\stackrel{\iota}{\longrightarrow}\,\text{At}(E_H)
\,\stackrel{d'p}{\longrightarrow}\, TX\, \longrightarrow\, 0\, ,
\end{equation}
which is known as the \textit{Atiyah exact sequence} \cite{At}; the above
homomorphism $d'p$ is given by $dp$ in \eqref{e-2}.

A $C^\infty$ \textit{connection} on $E_H$ compatible with its holomorphic structure is
a $C^\infty$ splitting of the exact sequence in \eqref{e5}. A \textit{holomorphic
connection} on $E_H$ is a holomorphic splitting of this exact sequence \cite{At}.

\subsection{Atiyah bundle for group action}

Let $G$ be a complex Lie group acting holomorphically on the left of
$X$ with
\begin{equation}\label{e6}
\rho\,:\, G\times X\, \longrightarrow\, X
\end{equation}
being the map giving the action. The Lie algebra of $G$ will be denoted by $\mathfrak g$. The
differential of $\rho$ in \eqref{e6} produces a ${\mathcal O}_X$--linear homomorphism
from the trivial vector bundle on $X$ with fiber $\mathfrak g$
\begin{equation}\label{e7}
d'\rho\, :\, X\times{\mathfrak g}\, \longrightarrow\, TX\, ,
\end{equation}
where ${\mathcal O}_X$ is the sheaf of holomorphic functions on $X$. We note that the
image of $d'\rho$ need not be a subbundle of $TX$.

Consider the holomorphic homomorphism of vector bundles
$$
\rho'\, :\, \text{At}(E_H)\oplus (X\times{\mathfrak g})\,\longrightarrow\, TX\, ,
\ \ (v\, ,w)\, \longmapsto\, d'p(v)-d'\rho(w)\, ,
$$
where $d'p$ and $d'\rho$ are constructed in \eqref{e5} and \eqref{e7} respectively. We
note that $\rho'$ is surjective because $d'p$ is surjective. Define the subsheaf
\begin{equation}\label{e8}
\text{At}_\rho(E_H)\, :=\,
(\rho')^{-1}(0)\, \subset\, \text{At}(E_H)\oplus (X\times{\mathfrak g})
\end{equation}
which is in fact a subbundle because $\rho'$ is surjective.

We have two homomorphisms
$$
\iota_0\, :\, \text{ad}(E_H)\,\longrightarrow\, \text{At}_\rho(E_H)\, , \ \ v\,
\longmapsto\, (\iota(v)\, ,0)\, ,
$$
where $\iota$ is constructed in \eqref{e5}, and
$$
q\,:\, \text{At}_\rho(E_H) \,\longrightarrow\, X\times{\mathfrak g}
\, , \ \ (v\, ,w)\, \longmapsto\, w\, ,
$$
where $v\,\in\, \text{At}(E_H)$ and $w\,\in\, X\times{\mathfrak g}$.
Consequently, there is a short exact sequence of holomorphic vector bundles on $X$
\begin{equation}\label{e9}
0\,\longrightarrow\, \text{ad}(E_H)\,\stackrel{\iota_0}{\longrightarrow}\,\text{At}_\rho
(E_H)\,\stackrel{q}{\longrightarrow}\, X\times{\mathfrak g}\, \longrightarrow\, 0\, .
\end{equation}

A holomorphic splitting of \eqref{e9} is a holomorphic homomorphism of vector bundles
$$
h\, :\, X\times{\mathfrak g}\,\longrightarrow\,\text{At}_\rho(E_H)
$$
such that $q\circ h\,=\, \text{Id}_{X\times{\mathfrak g}}$.

\begin{definition}\label{def1}
A \textit{holomorphic} $G$--\textit{connection} on $E_H$ is a holomorphic splitting 
of \eqref{e9}.
\end{definition}

A holomorphic section of $\text{At}(E_H)$ defined over an open subset
$U\, \subset\, X$ is a
$H$--invariant holomorphic vector field on $p^{-1}(U)$, where $p$
is the projection in \eqref{e-1}. The Lie bracket of two $H$--invariant holomorphic
vector fields on $p^{-1}(U)$ is again a $H$--invariant holomorphic
vector field on $p^{-1}(U)$. Therefore, the sheaf of holomorphic sections of
$\text{At}(E_H)$ has the structure of a Lie algebra.
This and the complex Lie algebra structure of $\mathfrak g$ together
produce a complex Lie algebra structure on the sheaf of holomorphic sections of
$\text{At}_\rho (E_H)$.

Since the adjoint action of $H$ on $\mathfrak h$ preserves its Lie algebra structure,
every fiber of the vector bundle $\text{ad}(E_H)$ in \eqref{f1} is a Lie algebra
isomorphic to $\mathfrak h$.
We note that the
above Lie algebra structure on the fibers of $\text{ad}(E_H)$ coincides with the
one given by the Lie bracket of sections of $T_{\rm rel}$ (see \eqref{e4} for $T_{\rm rel}$).

The homomorphism $\iota_0$ in \eqref{e9} is compatible with
the Lie bracket operations on the sections of $\text{ad}(E_H)$ and $\text{At}_\rho (E_H)$.
Similarly, the homomorphism $q$ is also compatible with
the Lie bracket operations on the sections of $\text{At}_\rho (E_H)$ and
$X\times{\mathfrak g}$.

Let $h\, :\, X\times{\mathfrak g}\,\longrightarrow\,\text{At}_\rho(E_H)$ be a
holomorphic $G$--connection on $E_H$. For any two holomorphic sections $s$ and $t$
of the trivial holomorphic vector bundle $X\times{\mathfrak g}$ defined over an open
subset $U\, \subset\, X$, consider
$$
{\mathcal K}(h)(s,t)\, :=\,
[h(s)\, ,h(t)] - h([s\, ,t])\, \in\, \Gamma(U,\, \text{At}_\rho(E_H))\, .
$$
Since the homomorphism $q$ in \eqref{e9} is compatible with the Lie algebra
structures, it follows that $q({\mathcal K}(h)(s,t))\,=\,0$.
Hence ${\mathcal K}(h)(s,t)$ lies in the image of $\text{ad}(E_H)$. We have
$$
{\mathcal K}(h)(fs,t)\, =\, f {\mathcal K}(h)(s,t)
$$
for any holomorphic function $f$ defined on $U$. Also, clearly
we have $${\mathcal K}(h)(s,t)\, =\,- {\mathcal K}(h)(t,s)\, .$$
Combining all these it follows that
\begin{equation}\label{b}
{\mathcal K}(h)\, \in\, H^0(X,\, \text{ad}(E_H)\otimes \bigwedge\nolimits^2
(X\times{\mathfrak g})^*)\, =\, H^0(X,\, \text{ad}(E_H))\otimes\bigwedge\nolimits^2
{\mathfrak g}^*\, .
\end{equation}
This section ${\mathcal K}(h)$ will be called the \textit{curvature} of the
holomorphic $G$--connection $h$ on $E_H$.

\subsection{Criterion for connection}

Henceforth, we will always assume that the complex manifold $X$ is compact.

The space of all infinitesimal deformations of the holomorphic principal $H$--bundle
$E_H$ are parametrized by $H^1(X,\, {\rm ad}(E_H))$. Therefore, given any holomorphic
principal $H$--bundle $\widetilde{E}_H$ on $T\times X$ with $T$ being a complex manifold,
and a holomorphic isomorphism of $E_H$ with $\widetilde{E}_H\vert_{\{t\}\times X}$, where
$t\, \in\, T$ is a fixed point, we have the infinitesimal deformation homomorphism
$$
T_tT\, \longrightarrow\, H^1(X,\, {\rm ad}(E_H))\, ,
$$
where $T_tT$ is the fiber at $t$ of the holomorphic tangent bundle $TT$ of $T$.

Let $\rho^*E_H\, \longrightarrow\, G\times X$ be the pulled back holomorphic 
principal $H$--bundle, where $\rho$ is the map in \eqref{e6}. Considering it as a 
holomorphic family of principal $H$--bundles on $X$ parametrized by $G$, we have the
infinitesimal deformation homomorphism
\begin{equation}\label{mu}
\mu\, :\, {\mathfrak g} \,=\, T_eG\, \longrightarrow\, H^1(X,\, {\rm ad}(E_H))\, .
\end{equation}

\begin{lemma}\label{lem1}
The principal $H$--bundle $E_H$ admits a holomorphic $G$--connection if and only if $\mu\,=\, 0$.
\end{lemma}

\begin{proof}
Let
\begin{equation}\label{mu2}
H^0(X,\, \text{At}_\rho(E_H))\, \stackrel{\mu_1}{\longrightarrow}\, 
H^0(X,\, X\times{\mathfrak g})\,=\, {\mathfrak g}\, \stackrel{\mu_2}{\longrightarrow}\,
H^1(X,\, {\rm ad}(E_H))
\end{equation}
be the long exact sequence of cohomologies associated to the short exact sequence
in \eqref{e9}.

We will show that the above homomorphism $\mu_2$ coincides with $\mu$ in \eqref{mu}.
This requires recalling the construction of $\mu$. Take any $v\, \in\, {\mathfrak g}$.
Let $\widetilde v$ be a holomorphic vector field defined around the identity
element $e\, \in\, G$ such that ${\widetilde v}(e)\,=\, v$. Take open subsets
$\{U_i\}_{i\in I}$ of $G\times X$ such that
\begin{enumerate}
\item $\{e\}\times X\, \subset\, \bigcup_{i\in I} U_i$, and

\item for each $i\,\in\, I$, the vector field $({\widetilde v}\, ,0)$ on $U_i$ lifts
to a $H$--invariant vector field ${\widetilde v}_i$ on $(\rho^*E_H)\vert_{U_i}$;
we choose such a vector field for each $i\,\in\, I$. (Here $0$ denotes the zero vector
field on $X$.)
\end{enumerate}
Now for each ordered pair $(i\, ,j)\, \in\, I\times I$, consider the vector
field $${\widetilde v}_i-{\widetilde v}_j$$ on $p^{-1}(U_i\bigcap U_j\bigcap (\{e\}\times X))
\,\subset\, E_H$, where $p$ is the projection in \eqref{e-1}.
They form a $1$--cocycle with values in ${\rm ad}(E_H)$. The corresponding cohomology
class in $H^1(X,\, {\rm ad}(E_H))$ is $\mu (v)$. From this it is straight--forward
to check that $\mu$ coincides with $\mu_2$ in \eqref{mu2}.

First assume that $\mu_2\,=\, 0$. So $\mu_1$ in \eqref{mu2} is surjective. Fix a complex
linear subspace $S\, \subset\, H^0(X,\, \text{At}_\rho(E_H))$ such that the restriction
$$
\mu_0\, :=\, \mu_1\vert_S\, : \, S\, \longrightarrow\, {\mathfrak g}
$$
is an isomorphism. Now define
$$
h\, :\, X\times{\mathfrak g}\, \longrightarrow\,\text{At}_\rho(E_H)\, ,\ \
(x\, ,v)\, \longmapsto\, (\mu_0)^{-1}(v)(x)\, .
$$
Clearly, we have $q\circ h\,=\, \text{Id}_{X\times{\mathfrak g}}$, where $q$ is the
homomorphism in \eqref{e9}. Hence $h$ defines a holomorphic $G$--connection on $E_H$.

Conversely, let $h\, :\, X\times{\mathfrak g}\, \longrightarrow\,\text{At}_\rho(E_H)$
be a holomorphic $G$--connection on $E_H$. Let
\begin{equation}\label{hst}
h_*\, :\, H^0(X,\, X\times{\mathfrak g})\, \longrightarrow\,
H^0(X,\, \text{At}_\rho(E_H))
\end{equation}
be the homomorphism induced by $h$. Since $\mu_1\circ h_*\,=\,
\text{Id}_{H^0(X,X\times{\mathfrak g})}$, where $\mu_1$ is the homomorphism in
\eqref{mu2}, it follows that $\mu_1$ is surjective. Hence from the exactness of
\eqref{mu2} we conclude that $\mu_2\,=\, 0$.
\end{proof}

\subsection{Homomorphisms and induced connection}

Let
$$
f\, :\, G_1\,\longrightarrow\, G
$$
be a holomorphic homomorphism of complex Lie groups. Using $f$, the action of $G$
on $X$ produces an action of $G_1$ on $X$. More precisely,
$$
\rho_1\, :\, G_1\times X \,\longrightarrow\, X\, ,\ \ (g\, ,x)\, \longmapsto\,
\rho(f(g)\, ,x)\, ,
$$
where $\rho$ is the map in \eqref{e6}, is a holomorphic action of $G_1$ on $X$.
The Lie algebra of $G_1$ will be denoted by ${\mathfrak g}_1$. Let
$$
df\, :\, {\mathfrak g}_1\,\longrightarrow\, {\mathfrak g}
$$
be the homomorphism of Lie algebras associated to $f$. From the construction of
$\text{At}_\rho(E_H)$ in \eqref{e8} it follows that
\begin{equation}\label{at1}
\text{At}_{\rho_1}(E_H)\,=\, \{(y\, ,z)\, \in\, \text{At}_\rho(E_H)\oplus
(X\times{\mathfrak g}_1)\, \mid\, q(y)\,=\, (\text{Id}_X\times df)(z)\}\, ,
\end{equation}
where $q$ is the homomorphism in \eqref{e9}.

\begin{lemma}\label{lem2}
A holomorphic $G$--connection $h$ on $E_H$ induces a holomorphic $G_1$--connection $h_1$ on $E_H$.
The curvature ${\mathcal K}(h_1)$ coincides with the image of ${\mathcal K}(h)$
under the homomorphism
$$
H^0(X,\, {\rm ad}(E_H))\otimes\bigwedge\nolimits^2
{\mathfrak g}^*\,\longrightarrow\, H^0(X,\, {\rm ad}(E_H))\otimes\bigwedge\nolimits^2
{\mathfrak g}_1^*
$$
given by the identity map of $H^0(X,\, {\rm ad}(E_H))$ and the homomorphism
$$
\bigwedge\nolimits^2
{\mathfrak g}^*\,\longrightarrow\, \bigwedge\nolimits^2
{\mathfrak g}_1^*
$$
induced by the dual homomorphism $(df)^*\, :\,
{\mathfrak g}^*\,\longrightarrow\, {\mathfrak g}^*_1$.
\end{lemma}

\begin{proof}
Consider the description of $\text{At}_{\rho_1}(E_H)$ in \eqref{at1}. Let
$$
h_1\, :\, X\times{\mathfrak g}_1\, \longrightarrow\, \text{At}_{\rho_1}(E_H)
$$
be the homomorphism defined by $$z\, \longmapsto\, (h((\text{Id}_X\times df)(z))\, ,
z)\, \in\, \text{At}_\rho(E_H)\oplus (X\times{\mathfrak g}_1)\, .$$
Then $h_1$ is a holomorphic $G_1$--connection on $E_H$. Its curvature ${\mathcal K}(h_1)$ 
is as described in the lemma.
\end{proof}

\section{Connection and lifting an action}

As before, $X$ is equipped with a holomorphic action of $G$, and $E_H$ is a holomorphic
principal $H$--bundle on $X$.

Let $\text{Aut}(E_H)$ denote the group of all holomorphic automorphisms of the
principal $H$--bundle $E_H$ over the identity map of $X$. In other words, an
element $g\, \in\, \text{Aut}(E_H)$ is a biholomorphism $E_H\,
\stackrel{g}{\longrightarrow}\, E_H$ such that
\begin{enumerate}
\item{} $p\circ g\,=\, p$, where $p$ is the
projection in \eqref{e-1}, and

\item{} $\psi(g(z)\, ,y)\,=\,g(\psi(z\, ,y))$ for
all $(z\, ,y)\,\in\, E_H\times H$, where $\psi$ is the action in \eqref{e3}.
\end{enumerate}
This $\text{Aut}(E_H)$ is a complex Lie group. Its Lie algebra is
$H^0(X,\, \text{ad}(E_H))$; the Lie algebra structure on the fibers of $\text{ad}(E_H)$
produces a complex Lie algebra structure on $H^0(X,\, \text{ad}(E_H))$.

Consider the action $\rho$ in \eqref{e6}. For any $z\, \in\, G$, let
\begin{equation}\label{rz}
\rho_z\, :\, X\, \longrightarrow\,X
\end{equation}
be the holomorphic automorphism defined by $x\, \longmapsto\, \rho(z\, ,x)$.

Let
\begin{equation}\label{g1}
G_1\, \subset\, G
\end{equation}
be the subset consisting all $z\, \in\, G$ such that
the pulled back principal $H$--bundle $\rho^*_z E_H$ is holomorphically isomorphic
to $E_H$ over the identity map of $X$. So $z\, \in\, G_1$ if and only only if
there is a holomorphic automorphism of the principal $H$--bundle $E_H$ over
the automorphism $\rho_z$ of $X$.
Let $\mathcal G$ denote the space of all pairs of the form $(y\, , z)$, where $z\,\in\,
G_1$, and $$y\, :\, E_H \,\longrightarrow\, E_H$$ is a holomorphic automorphism of
the principal $H$--bundle over the automorphism $\rho_z$ of $X$. We observe that
$\mathcal G$ is equipped with the group operation defined by
$$(y'\, , z')\cdot (y\, , z)\,=\, (y'\circ y\, , z'z)\, ,$$ while the inverse
is the map $(y\, , z)\, \longmapsto\, (y^{-1}\, , z^{-1})$. Therefore,
$\mathcal G$ fits in the following short exact sequence of groups
\begin{equation}\label{e10}
0\,\longrightarrow\, \text{Aut}(E_H)\,\stackrel{\alpha}{\longrightarrow}\,
{\mathcal G}\,\stackrel{\beta}{\longrightarrow}\, G_1 \,\longrightarrow\, 0\, ,
\end{equation}
where $\beta(y\, , z)\,=\, z$ and $\alpha (t)\,=\, (t\, ,e)$ with $e$ being the
identity element of $G_1$. There is a complex Lie group structure on $\mathcal G$
which is uniquely determined by the condition that \eqref{e10} is a sequence
of complex Lie groups.

The Lie algebra structure on the sheaf of sections of $\text{At}_\rho (E_H)$ produces 
the structure of a complex Lie algebra on $H^0(X,\, \text{At}_\rho (E_H))$.

\begin{proposition}\label{prop1}
The Lie algebra of $\mathcal G$ is canonically identified with the above Lie
algebra $H^0(X,\, {\rm At}_\rho (E_H))$.
\end{proposition}

\begin{proof}
The Lie algebra of $\mathcal G$ will be denoted by $\widetilde{\mathfrak g}$. We will show
that there is a natural homomorphism from $\widetilde{\mathfrak g}$ to
$H^0(X,\, \text{At}_\rho (E_H))$.

First observe that the group $\mathcal G$ has a tautological action on the total space
$E_H$. Indeed, the action of $(y\, , z)\,\in\, \mathcal G$ sends any $x\, \in\, E_H$
to $y(x)\,\in\, E_H$. It is straight-forward to check that this defines a holomorphic
action of $\mathcal G$ on $E_H$. From the definition of $\mathcal G$ it follows
immediately that this action of $\mathcal G$ commutes with the
action of $H$ on $E_H$. Consequently, we get a homomorphism of complex Lie algebras
$$
h'\,:\, \widetilde{\mathfrak g}\, \longrightarrow\, H^0(X,\, \text{At}(E_H))\, .
$$
Now define
\begin{equation}\label{h1}
h_1\,:\, \widetilde{\mathfrak g}\, \longrightarrow\, H^0(X,\, \text{At}_\rho(E_H))\, ,\ \
v\, \longmapsto\, (h'(v)\, , d\beta(v))\, \in\,H^0(X,\, \text{At}(E_H))
\oplus{\mathfrak g}\, ,
\end{equation}
where $d\beta\, :\, \widetilde{\mathfrak g}\,\longrightarrow\, \text{Lie}(G_1)\,
\hookrightarrow\, \mathfrak g$
is the homomorphism of Lie algebras associated to $\beta$ in \eqref{e10}; it is
straight-forward to check that
$$
(h'(v)\, , d\beta(v))\, \in\,H^0(X,\, \text{At}_\rho (E_H))\, \subset\,
H^0(X,\, \text{At}(E_H))\oplus{\mathfrak g}\, ;
$$
see \eqref{e8}. Clearly, $h_1$ is an injective homomorphism of complex Lie algebras.

To prove that $h_1$ is surjective, take any $w\, \in\, H^0(X,\, \text{At}(E_H))$.
Let $t\, \longmapsto\, \varphi^t_w$, $t\, \in\, \mathbb C$, be the $1$--parameter family
of biholomorphisms of $E_H$ associated to $w$. We note that $\varphi^t_w$ exists because
$X$ is compact and $w$ is $H$--invariant. Since $w$ is fixed by the action of $H$ on $E_H$,
it follows immediately that the biholomorphism $\varphi^t_w$ commutes with the action of
$H$ on $E_H$ for every $t$.

Now assume that there is an element $v\, \in\,\mathfrak g$ such that
$$
(w\, ,v)\, \in\, H^0(X,\, \text{At}_\rho (E_H))\, \subset\,
H^0(X,\, \text{At}(E_H))\oplus{\mathfrak g}
$$
(see \eqref{e8}). Let $t\, \longmapsto\, \exp(t v)$, $t\, \in\, \mathbb C$, be the
$1$--parameter subgroup of $G$ associated to $v$. Now we observe that $(\varphi^t_w\, ,
\exp(t v))\, \in\, {\mathcal G}$. Indeed, this follows from the fact that the
vector field on $E_H$ given by $w$ projects to the vector field on $X$ given by $v$.
Consequently, the above element $(w\, ,v)\, \in\,
H^0(X,\, \text{At}_\rho (E_H))$ lies in the image of the homomorphism $h_1$ in
\eqref{h1}. Hence $h_1$ is surjective.
\end{proof}

In the proof of Proposition \ref{prop1} we saw that
${\mathcal G}$ has a tautological action on $E_H$. Let
$$
\eta\, :\, {\mathcal G}\times E_H\,\longrightarrow\, E_H
$$
be this action.

\begin{lemma}\label{lem3}
There is a natural holomorphic isomorphism of vector bundles
$$
{\rm At}_\eta (E_H)\,\longrightarrow\, {\rm ad}(E_H)
\oplus (X\times H^0(X,\, {\rm At}_\rho (E_H)))\, ,
$$
where $X\times H^0(X,\, {\rm At}_\rho (E_H))$ is the trivial vector bundle
on $X$ with fiber $H^0(X,\, {\rm At}_\rho (E_H))$, and
${\rm At}_\eta (E_H)$ is constructed as in \eqref{e8}.
\end{lemma}

\begin{proof}
Since $\text{At}_\eta (E_H)$ has a natural projection to $X\times\widetilde{\mathfrak g}$
(see \eqref{e9}), and Proposition \ref{prop1} identifies $\widetilde{\mathfrak g}$ with
$H^0(X,\, {\rm At}_\rho (E_H))$, we obtain a projection
$$
q_1\, :\, {\rm At}_\eta (E_H)\,\longrightarrow\, X\times H^0(X,\, {\rm At}_\rho (E_H))\, .
$$

The action of $\mathcal G$ on $X$ factors through the action of $G$ on $X$.
Recall the description of ${\rm At}_\eta (E_H)$ given in \eqref{at1}. Consider
the projection
$$
q'\, :\, \text{At}_\rho(E_H)\oplus(X\times\widetilde{\mathfrak g})\,=\,\text{At}_\rho
(E_H)\oplus (X\times H^0(X,\, {\rm At}_\rho (E_H)))\,\longrightarrow\, \text{At}_\rho(E_H)
$$
that sends any $(z\, ,(x\, ,s))$, where $x\, \in\, X$, $z\, \in\, \text{At}_\rho(E_H)_x$
and $s\, \in\, H^0(X,\, {\rm At}_\rho (E_H))$, to $z-s(x)\,\in\, \text{At}_\rho(E_H)_x$.
{}From \eqref{at1} it follows immediately, that
$$q\circ (q'\vert_{{\rm At}_\eta (E_H)})\,=\, 0\, ,$$ where $q$ is the projection
in \eqref{e9}. Therefore, the restriction $q'\vert_{{\rm At}_\eta (E_H)}$
produces a homomorphism
\begin{equation}\label{q2}
q_2\, :\, {\rm At}_\eta (E_H)\, \longrightarrow\, \text{kernel}(q)\,=\,
\text{ad}(E_H)\, .
\end{equation}
Now it is straight-forward to check that the composition
$$
q_2\oplus q_1\, :\, {\rm At}_\eta (E_H)\,\longrightarrow\, {\rm ad}(E_H)
\oplus (X\times H^0(X,\, {\rm At}_\rho (E_H)))\, ,
$$
is an isomorphism.
\end{proof}

\begin{proposition}\label{prop2}
The principal $H$--bundle $E_H$ has a tautological holomorphic $\mathcal G$--connection.
The curvature of this holomorphic $\mathcal G$--connection on $E_H$ vanishes identically.
\end{proposition}

\begin{proof}
Let $\iota'\, :\, {\rm ad}(E_H)\,\longrightarrow\,
{\rm At}_\eta (E_H)$ be the natural inclusion (see \eqref{e9}). For the homomorphism
$q_2$ in \eqref{q2}, we have
$$
q_2\circ \iota'\,=\, \text{Id}_{{\rm ad}(E_H)}\, .
$$
Therefore, $\text{kernel}(q_2)$ provides a holomorphic splitting of the analog of the short
exact sequence \eqref{e9} for ${\rm At}_\eta (E_H)$. In other words, $\text{kernel}(q_2)$
defines a holomorphic $\mathcal G$--connection on $E_H$.

Since the homomorphism $q_2$ preserves the Lie algebra structure on the sheaf of sections of
${\rm At}_\eta (E_H)$ and ${\rm ad}(E_H)$, it follows that the sheaf of sections of
$\text{kernel}(q_2)$ is closed under the Lie algebra structure on the sheaf of sections of
${\rm At}_\eta (E_H)$. Therefore, the curvature of the holomorphic $\mathcal G$--connection on $E_H$
defined by $\text{kernel}(q_2)$ vanishes identically.
\end{proof}

\section{Equivariant bundles and connection}

Now-onwards, we assume that the Lie group $G$ is connected.

An equivariance structure on the principal $H$--bundle $E_H$ is a holomorphic action 
of $G$ on the total space of $E_H$
$$
\rho_E\, :\, G\times E_H\, \longrightarrow\, E_H
$$
such that
\begin{enumerate}
\item{} $p\circ\rho_E\,=\, \rho\circ (\text{Id}_G\times p)$, where $p$ and $\rho$
are the maps in \eqref{e-1} and \eqref{e6} respectively, and

\item{} $\rho_E\circ (\text{Id}_G\times \psi)\,=\, \psi\circ (\rho_E\times \text{Id}_H)$
as maps from $G\times E_H\times H$ to $E_H$, where $\psi$ is the action in \eqref{e3}.
\end{enumerate}
An equivariant principal $H$--bundle is a principal $H$--bundle with an
equivariance structure.

\begin{lemma}\label{lem4}
Let $(E_H\, , \rho_E)$ be an equivariant principal $H$--bundle. Then $E_H$ has a
tautological holomorphic $G$--connection. The curvature of this holomorphic $G$--connection
vanishes identically.
\end{lemma}

\begin{proof}
Note that for any $g\, \in\, G$, the map
$$
\rho^g_E\, :\, E_H\, \longrightarrow\, E_H\, , \ \ \ z\,\longmapsto\, \rho_E(g\, ,z)\, ,
$$
is an automorphism of the principal $H$--bundle
$E_H$ over the automorphism $\rho_g$ of $X$ in \eqref{rz}. Therefore, the group $G_1$
in \eqref{g1} coincides with $G$. In fact, the above map
$$
g\, \longrightarrow\, \rho^g_E
$$
is a homomorphism
$$
\beta_E\, :\, G\, \longrightarrow\, \mathcal G
$$
such that $\beta\circ\beta_E \,=\, \text{Id}_G$, where $\beta$ is the homomorphism
in \eqref{e10}.

Consider the tautological holomorphic $\mathcal G$--connection in Proposition \ref{prop2}.
In view of the above homomorphism $\beta_E$, using Lemma \ref{lem2} it produces
a holomorphic $G$--connection. From Lemma \ref{lem2} it also follows that the curvature of
this holomorphic $G$--connection vanishes identically.
\end{proof}

The following is a converse of Lemma \ref{lem4}.

\begin{lemma}\label{lem5}
Let $h\, :\, X\times{\mathfrak g}\, \longrightarrow\, {\rm At}_\rho (E_H)$ be
a holomorphic $G$--connection on $E_H$ such that the curvature vanishes identically. Assume
that $G$ is simply connected. Then there is an equivariance structure 
$$
\rho_E\, :\, G\times E_H\, \longrightarrow\, E_H
$$
such that the holomorphic $G$--connection associated to it by Lemma \ref{lem4} 
coincides with $h$.
\end{lemma}

\begin{proof}
Recall that $H^0(X,\, \text{At}_\rho (E_H))\,=\, \widetilde{\mathfrak g}
\,=\, \text{Lie}({\mathcal G})$ by Proposition \ref{prop1}. Let
$$
h_*\ :\, {\mathfrak g}\,=\, H^0(X,\, X\times{\mathfrak g})
\, \longrightarrow\, H^0(X,\, \text{At}_\rho (E_H))\,=\, \widetilde{\mathfrak g}
$$
be the $\mathbb C$--linear map induced by $h$. Since the curvature of the
holomorphic $G$--connection $h$ vanishes identically, it follows that $h_*$ is a homomorphism
of Lie algebras. As $G$ is simply connected, there is a unique holomorphic homomorphism
of complex Lie groups
$$
\gamma\,:\, G\, \longrightarrow\, {\mathcal G}
$$
such that the differential $d\gamma(e)\, :\, {\mathfrak g}\, \longrightarrow\,
\widetilde{\mathfrak g}$ coincides with $h_*$. Now $\gamma$ produces an
equivariance structure on $E_H$; recall that $\mathcal G$ acts on $E_H$. The corresponding
holomorphic $G$--connection given by Lemma \ref{lem4} clearly coincides with $h$.
\end{proof}

\subsection{The group $G$ is semisimple}

We now assume that $G$ is a semisimple and simply connected affine algebraic group
defined over $\mathbb C$.

The following theorem show that holomorphic $G$--connections produce $G$--equivariance structures.

\begin{theorem}\label{thm1}
Let $E_H$ be a holomorphic principal $H$--bundle on $X$ admitting a holomorphic $G$--connection
$h$. Then $E_H$ admits an equivariance structure
$$
\rho_E\, :\, G\times E_H\, \longrightarrow\, E_H\, .
$$
\end{theorem}

\begin{proof}
Consider the homomorphisms $\mu_1$ and $h_*$ constructed in \eqref{mu2}
and \eqref{hst} respectively. Since
$$
\mu_1\circ h_*\,=\,
\text{Id}_{H^0(X,X\times{\mathfrak g})}\,=\, \text{Id}_{\mathfrak g}\, ,
$$
we know that the Lie algebra homomorphism $\mu_1$ is surjective. As $\mathfrak g$
is semisimple, there is a Lie subalgebra
$${\mathcal S}\, \subset\, H^0(X,\, \text{At}_\rho(E_H))$$ such that the restriction
$$
\widehat{\mu}\, :=\, \mu_1\vert_{\mathcal S}\, :\,
{\mathcal S}\, \longrightarrow\, H^0(X,X\times{\mathfrak g})\,=\, {\mathfrak g}
$$
is an isomorphism \cite[p. 91, Corollaire 3]{Bo}. Fix a subspace $\mathcal S$ as above. Define
$\widehat{h}$ to be the composition
$$
H^0(X,X\times{\mathfrak g})\,=\,{\mathfrak g}\, \stackrel{\widehat{\mu}^{-1}}{\longrightarrow}
\,{\mathcal S}\, \hookrightarrow\, H^0(X,\, \text{At}_\rho(E_H))\, .
$$
Since $\widehat{\mu}$ is a homomorphism of Lie algebras, it follows that
$\widehat{h}$ is a holomorphic $G$--connection on $E_H$ such that the curvature vanishes 
identically. Now the theorem follows from Lemma \ref{lem5}.
\end{proof}

\begin{corollary}\label{cor1}
Let $E_H$ be a holomorphic principal $H$--bundle on $X$ such that the
pulled back principal $H$--bundle $\rho^*_z E_H$ is holomorphically
isomorphic to $E_H$ for every $z\, \in\, G$, where $\rho_z$ is defined in
\eqref{rz}. Then $E_H$ admits an equivariance structure.
\end{corollary}

\begin{proof}
Since $\rho^*_z E_H$ is holomorphically
isomorphic to $E_H$ for all $z\, \in\, G$, it follows that the infinitesimal
deformation map $\mu$ in \eqref{mu} is the zero homomorphism. Therefore,
$E_H$ admits a holomorphic $G$--connection by Lemma \ref{lem1}. Now Theorem \ref{thm1}
completes the proof.
\end{proof}

\section{Some examples}

\subsection{Trivial action}

Consider the trivial action of $G$ on $X$. So $\rho(g\, ,x)\,=\, x$ for all
$g\,\in\, G$ and $x\, \in\, X$. In this case 
$$
\text{At}_\rho (E_H)\,=\, \text{ad}(E_H)\oplus (X\times{\mathfrak g})\, .
$$
Therefore, a holomorphic $G$--connection on $E_H$ is a holomorphic homomorphism
$X\times{\mathfrak g}\,\longrightarrow\, \text{ad}(E_H)$. Note that there is a
tautological holomorphic $G$--connection on $E_H$ given by the zero homomorphism from
$X\times{\mathfrak g}$ to $\text{ad}(E_H)$.

\subsection{Trivial tangent bundle}

Let $X$ be a compact complex manifold such that the holomorphic tangent bundle
$TX$ is holomorphically trivial. Then a theorem of Wang says that there is a
complex connected Lie group $G$ and a discrete subgroup $\Gamma\,\subset\, G$,
such that $X$ is biholomorphic to $G/\Gamma$ \cite[p.~774, Theorem~1]{Wa}. Fix
an isomorphism of $X$ with $G/\Gamma$. Let $\rho$ be the
left translation action of $G$ on $G/\Gamma\,=\, X$. In this case, the homomorphism
$d'\rho$ in \eqref{e7} is an isomorphism. Therefore, the Atiyah exact sequence
in \eqref{e5} coincides with the exact sequence in \eqref{e9}.

Consequently, holomorphic $G$--connections on $E_H$ are same as holomorphic
connections on $E_H$.

\subsection{Smooth toric variety}

Let $X$ be a smooth complex projective toric manifold such that the subvariety $D\, \subset\, X$ 
where the torus action is not free is actually a simple normal crossing divisor. Then the 
logarithmic tangent bundle $TX(-\log D)$ is holomorphically trivial \cite[p.~317, 
Lemma~3.1(2)]{BDP}. Therefore, the holomorphic $G$--connections on $E_H$ are the logarithmic connections on 
$E_H$ singular over $D$. Hence \cite[p.~319, Proposition~3.2]{BDP} follows from Lemma
\ref{lem4}. We note that \cite[p.~324, Theorem~4.2]{BDP} can also be proved modifying the
proof of Theorem \ref{thm1}.

\subsection{Homogeneous manifolds}

Let $M$ be a closed subgroup of a complex connected Lie group $G$. The group $G$ acts
on $X\,:=\, G/M$ as left--translations. The left--translation action of $G$ on itself will
be denoted by $\theta$. Consider the short exact sequence
in \eqref{e9} over $X$. Pull this exact sequence back to $G$ by the quotient map
$q\, :\, G\, \longrightarrow\, G/M$. Note that this pullback is of the form
\begin{equation}\label{5.1}
0\,\longrightarrow\, \text{ad}(q^*E_H)\,{\longrightarrow}\,\text{At}_\theta
(q^* E_H)\,{\longrightarrow}\, G\times{\mathfrak g}\, \longrightarrow\, 0\, ,
\end{equation}
where $\mathfrak g$ is the Lie algebra of $G$. Observe that \eqref{5.1}
coincides with the Atiyah exact sequence for $q^*E_H$.

Now consider the right--translation action of $M$ on $G$. Since $q^*E_H$ is
pulled back from $G/M$, it follows that $q^*E_H$ is a $M$--equivariant
principal $H$--bundle on $G$.

{}From the above observation that \eqref{5.1} is the Atiyah exact sequence for $q^*E_H$
it follows that the holomorphic $G$--connections on $E_H$ are precisely the $M$--equivariant
holomorphic connections on the $M$--equivariant principal $H$--bundle on $q^*E_H$.


\end{document}